\documentclass[10pt]{amsart}
\usepackage{amssymb}
\usepackage{bbm}
\usepackage[arrow,matrix]{xy}

\numberwithin{equation}{section}
\newtheorem{theorem}{Theorem}[section]
\newtheorem{corollary}[theorem]{Corollary}

\newtheorem{proposition}[theorem]{Proposition}
\newtheorem{definition}[theorem]{Definition}
\newtheorem{remark}[theorem]{Remark}
\newtheorem{example}[theorem]{Example}

\makeatletter

\def\Hom{\operatorname {Hom}}

\def\HH{\operatorname {HH}}
\def\HP{\operatorname {HP}}

\title[Frobenius Poisson algebras]
{\bf  On  (co)homology of Frobenius Poisson algebras}

\author{Can  Zhu}
\address{College of Science, University of Shanghai for Science and Technology, Shanghai
200093, China}
\address{Department of Mathematics and Computer Science, University of Antwerp, Middelheimlaan 1,
B-2020 Antwerp, Belgium} \email{czhu@usst.edu.cn}

\author{Fred Van Oystaeyen}
\address{Department of Mathematics and Computer Science, University of Antwerp, Middelheimlaan 1,
B-2020 Antwerp, Belgium} \email{fred.vanoystaeyen@ua.ac.be}

\author{Yinhuo Zhang}
\address{Department Mathematics and Statistics, University of Hasselt, Universitaire Campus, 3590 Diepeenbeek, Belgium} \email{yinhuo.zhang@uhasselt.be}

\date{}

\begin{document}
\begin{abstract}
In this paper, we study Poisson (co)homology of a Frobenius Poisson algebra. More precisely, we show that there exists a duality between Poisson homology and Poisson cohomology of Frobenius Poisson algebras, similar to that between Hochschild homology and Hochschild cohomology of Frobenius algebras. Then we use the non-degenerated bilinear form on a unimodular Frobenius  Poisson algebra to construct a Batalin-Vilkovisky structure on the Poisson cohomology
ring making it into a Batalin-Vilkovisky algebra.
\end{abstract}



\maketitle

\section{Introduction}
In 1977, A. Lichnerowicz introduced Poisson cohomology of a Poisson algebra \cite{Lic77} (see also \cite{Hue90} for an algebraic approach). This cohomology provides valuable information about the Poisson structure, such as Casimir elements, Poisson derivations, Poisson deformations and so on. It has been shown in the smooth case that  Poisson (co)homology of a Poisson algebra  has similar properties to those of Hochschild (co)homology of a deformation quantization of a Poisson algebra \cite{HKR62, Bry88, Kas88, Kon03, LR07, Dol09}.  So a natural question is whether the same holds  in the singular case as well.
The aim of this paper is to establish a duality between Poisson homology and  Poisson cohomology of Poisson algebras in the singular case and to construct a Batalin-Vilkovisky structure on the Poisson cohomology ring, which makes it a Batalin-Vilkovisky algebra.

For a Frobenius algebra $A$ with a Nakayama automorphism $\nu$, it is known that there is an isomorphism between  Hochschild homology and Hochschild cohomology:
\begin{eqnarray}
\label{eq1} \HH^i(A, A)^*\cong\HH_i(A, A^{\nu}), \forall i\in \mathbb{N},
\end{eqnarray}
(see \cite[p.120, Proposition 5.1]{CE56}, for example);
here, $*$ stands for taking the dual over $\mathbbm{k}$
and $A^{\nu}$ is an $A$-bimodule with the right action twisted by $\nu$.  In  \cite{LR09},
Launois and Richard study Poisson (co)homology of truncated polynomial algebras in two variables and obtain an isomorphism between  Poisson homology and Poisson cohomology similar to that in \eqref{eq1}, see [LR09]. Motivated by \eqref{eq1} and the results in [LR09], we obtain an isomorphism between  Poisson homology and  Poisson cohomology of Frobenius Poisson algebras as follows (Corollary \ref{dual2}):

\noindent {\bf Theorem 1.} Let $S$ be a Frobenius Poisson algebra.
Then we have the following isomorphism:
$$\HP^{i}(S, \, S)^*\cong\HP_i(S, \,S_\sigma),$$
for all $i\in \mathbb{N}$, where $S_{\sigma}$ is the Poisson module induced by the Frobenius isomorphism $\sigma: S\rightarrow S^*$.

Note that a Frobenius Poisson algebra is always symmetric as an algebra. Thus, the Nakayama automorphism is the identity map. However, $S_\sigma$ is not necessarily isomorphic to $S$ as Poisson modules. Therefore, we introduce the notion of a {\it unimodular Frobenius
Poisson algebra}, that is, a Frobenius Poisson algebra such that $S_{\sigma}$ and $S$ are isomorphic (see Section $4$ for details).

For a Poisson algebra $S$, we denote by $\mathfrak{X}^k(S)$ the space of all skew-symmetric $\mathbbm{k}$-linear maps $S^{\otimes k}\rightarrow S$ that are derivations in each argument, and
let $\mathfrak{X}(S):=\oplus_k\mathfrak{X}^k(S)$. There is a cochain map on $\mathfrak{X}(S)$ and the corresponding
cohomology is called the Poisson cohomology and denoted by $\HP^\bullet(S)$. It is known that $\mathfrak{X}(S)$ and the Poisson cohomology ring $\HP(S):=\oplus\HP^\bullet(S)$ are Gerstenhaber algebras \cite{LPV13}.
Moreover, a Batalin-Vilkovisky structure can be constructed on the Poisson cohomology ring of a unimodular Poisson manifold \cite{Xu99}.
Also, a Batalin-Vilkovisky structure exists on the Hochschild cohomology ring of some algebras, such as Calabi-Yau algebras \cite{Gin06}, symmetric Frobenius algebras \cite{Tra08, Men04} and preprojective algebras of Dynkin quivers \cite{ES09, Eu10}.
In the current paper, we shall use the non-degenerated bilinear form on a Frobenius Poisson algebra to prove the following (Theorem \ref{combv} and Theorem \ref{pcombv}):

\noindent {\bf Theorem 2.}\begin{itemize}
\item[(1)] Let $S$ be a Frobenius Poisson algebra with the non-degenerated bilinear form $\langle-, -\rangle: S\otimes S\rightarrow \mathbbm{k}$.
For $P\in \mathfrak{X}^m(S) (m\geq 1)$, let $\Delta(P)\in \mathfrak{X}^{m-1}(S)$, defined by the equation:
$$\langle\Delta(P)(a_1, \cdots, a_{m-1}), a_m\rangle=(-1)^{m-1}\langle P(a_1, \cdots, a_{m-1}, a_m),1\rangle.$$
Then  $(\mathfrak{X}(S), \Delta)$ is a Batalin-Vilkovisky algebra.
\item[(2)] If $S$ is a unimodular Frobenius Poisson algebra, then $(\HP(S), \Delta)$ is a Batalin-Vilkovisky algebra.
\end{itemize}

The paper is organized as follows. In Section 2, we recall some basic facts on
Poisson algebras, Poisson modules and Poisson (co)homology. In Section 3, we prove Theorem 1. In the last section,
we introduce the notion of a unimodular Poisson algebra and construct a Batalin-Vilkovisky structure on the Poisson cohomology ring of a unimodular Poisson algebra.

Throughout the paper, $\mathbbm{k}$  is a base field and all algebras are $\mathbbm{k}$-algebras; $^*$ means the dual over $\mathbbm{k}$ and the unadorned $\otimes$ stands for $\otimes_\mathbbm{k}$.

\section{Preliminaries}
This section deals with  some basic facts on
Poisson algebras, Poisson modules and Poisson (co)homology.

By definition a {\it Poisson algebra} is a commutative $\mathbbm{k}$-algebra
$S$ equipped with a bilinear map $\{-, -\}: S\otimes S\rightarrow S$
(Poisson bracket) satisfying:
\begin{itemize}
\item $\{-, -\}$ is a Lie bracket,

\item $\{-, -\}$ is a derivation in each variable.
\end{itemize}

An algebra $E$ is called a Frobenius algebra (not necessarily commutative) if there is an isomorphism $\sigma: E\rightarrow E^*$ as left
modules. Equivalently, there is a nondegenerate bilinear form, often
called Frobenius pair, $\langle-, -\rangle: E\times E\rightarrow \mathbbm{k}$
such that $\langle ab, c\rangle=\langle a, bc\rangle$ for all $a, b ,c\in E$ (where the bilinear form is defined by
$\langle a, b\rangle:=\sigma(b)(a)$). The isomorphism $\sigma$ is often called a Frobenius isomorphism.
By the nondegeneracy of the bilinear form, there exists an automorphism
$\nu$, unique up to an inner automorphism, such that
$$\langle a, b\rangle=\langle\nu(b), a\rangle$$ for all $a, b\in E$.
Thus, $\sigma$ becomes an isomorphism of $E$-bimodules $E^\nu\rightarrow E^*$.
The automorphism $\nu$ is called the {\it Nakayama automorphism}
of $E$.  For more detail, see \cite{Sm96}. The Frobenius Poisson algebras considered in this paper  are defined as follows:

\begin{definition} A Poisson algebra $S$ is called a Frobenius Poisson algebra if $S$ as an associative algebra is Frobenius.
\end{definition}

\begin{definition} \cite{Oh99}  A right Poisson module over a Poisson
algebra $S$ is a $\mathbbm{k}$-vector space $M$ endowed with two bilinear maps
$\cdot$ and $\{-, -\}_M: M\otimes S\rightarrow M$ such that
\begin{itemize}
\item[(1)] $(M, \cdot)$ is a right module over the commutative algebra $S$,

\item[(2)] $(M, \{-, -\}_M)$ is a right module over the Lie algebra $(S, \{-, -\})$,

\item[(3)] $\{xa, b\}_M=\{x, b\}_M\cdot a+x\cdot\{a, b\}$ for all $a, b\in S$ and $x\in M$,

\item[(4)] $\{x, ab\}_M=\{x, a\}_M\cdot b+\{x, b\}_M\cdot a$ for all $a, b\in S$ and $x\in M$.
\end{itemize}
\end{definition}

The notion of a left Poisson module is defined in a similar way. For a Poisson algebra $S$, the space $S$ has a
natural right (also, left) Poisson module structure.
Given two right Poisson modules $(M, \cdot, \{-, -\}_M)$ and $(N, \cdot, \{-, -\}_N)$ over $S$, a $\mathbbm{k}$-linear map
$f: M\rightarrow N$  is called
a morphism of Poisson modules if $$f(m\cdot s)=f(m)\cdot s, \quad f(\{m, s\}_M)=\{f(m), s\}_N$$
for each $m\in M$ and $s\in S$. The following two properties on Poisson modules are straightforward.

\begin{proposition} \label{dual} Suppose that $(M, \cdot, \{-, -\}_M)$ is a right  Poisson module over $S$.
Then the following actions define a left Poisson module structure on $M^*:=\Hom_\mathbbm{k}(M, \mathbbm{k})$
\begin{itemize}
\item[$\bullet$] $s\cdot \alpha: M\rightarrow \mathbbm{k}$ by $(s\cdot \alpha)(m):=\alpha(m\cdot s)$,

\item[$\bullet$] $\{s, \alpha\}_{M^*}: M\rightarrow \mathbbm{k}$ by $\{s, \alpha\}_{M^*}(m):=\alpha(\{m, s\}_{M})$
\end{itemize}
for each $s\in S$, $\alpha\in M^*$  and $m\in M$. Similarly, a left Poisson module $M$ yields a right Poisson module $M^*$.
\end{proposition}

\begin{proposition} \label{iso} Let $(N, \cdot, \{-, -\}_N)$ be a right  Poisson module over $S$ and $M$ be
a right $S$-module. Suppose that $f: M\rightarrow N$ is an isomorphism of $S$-modules. Then there exists a right Poisson module structure on $M$ given by:
$$ \{m, s\}_{M}:=f^{-1}(\{f(m), s\}_{N}) $$
for each $s\in S$ and $m\in M$, such that $f$ is an isomorphism of Poisson modules.
\end{proposition}

Now, we recall the definition of  Poisson homology and  Poisson cohomology of a Poisson algebra. Let $S$ be a Poisson algebra.  Write $\Omega^1(S)$ for the $S$-module of K\"{a}hler differentials of $S$, and let
$\Omega^\bullet(S):=\bigwedge^\bullet_S(\Omega^1(S))$ be the graded commutative
algebra of differential forms, equipped with the de Rham differential $\mathbf{d}$:
$$\mathbf{d}(a_0da_1\wedge\cdots\wedge da_k)=da_0\wedge da_1\wedge\cdots\wedge da_k.$$
Let $M$ be a right Poisson module over $S$. Then there
is a chain complex on the $S$-module $M\otimes_S \Omega^k(S)$. The boundary
operator $\partial_k: M\otimes_S \Omega^k(S)\rightarrow
M\otimes_S \Omega^{k-1}(S)$ is defined by:
\begin{align*}
\partial_k(m\otimes & da_1\wedge\cdots \wedge da_k) = \sum_{1\leq i\leq k}(-1)^{i+1}\{m, a_i\}_M\otimes da_1\wedge\cdots\wedge \widehat{da_i}\wedge\cdots \wedge da_k +\\
 & +\sum_{1\leq i< j\leq k}(-1)^{i+j}m\otimes d\{a_i, a_j\}\wedge da_1\wedge\cdots \wedge\widehat{da_i}\wedge\cdots \wedge\widehat{da_j}\wedge\cdots \wedge da_k,
\end{align*}
where $d$ denotes the exterior differential. It is easy to see that
$\partial_k$ is well defined and that $\partial_{k-1}\partial_k=0$.
The homology of this complex is denoted by $\HP_\bullet(S, M)$ and is
called the Poisson homology of $S$ with coefficients in $M$.

Denote by $\mathfrak{X}^k(S)$ the space of all skew-symmetric
$\mathbbm{k}$-linear maps $S^{\otimes k}\rightarrow S$ that are derivations in each argument. Here, a map $f: S^k\rightarrow S$ is called
skew-symmetric if $f(a_1, \cdots,
a_k)=\varepsilon(\theta)f(a_{\theta_1}, \cdots,
a_{\theta_k})$ for any permutation $\theta\in S_k$, where
$\varepsilon(\theta)$ denotes its sign and $\theta_i$ stands for
$\theta(i)$. There is a cochain complex $(\mathfrak{X}^\bullet(S),\delta_\bullet)$, where $\delta_k: \mathfrak{X}^k(S)\rightarrow
\mathfrak{X}^{k+1}(S)$ is defined by:
\begin{align*}
\delta_k(P)(y_0, y_1, \cdots, y_k) &:= \sum_{0\leq i\leq k}(-1)^i\{y_i, P(y_0, y_1, \cdots, \widehat{y_i}, \cdots y_k)\} +\\
 & +\sum_{0\leq i< j\leq k}(-1)^{i+j}P(\{y_i, y_j\}, y_0, y_1, \cdots, \widehat{y_i}, \cdots\widehat{y_j}, \cdots, y_k)
\end{align*}
for all $P\in \mathfrak{X}^k(S)$. It is  not hard to check that
$\delta_k(P)$ belongs to $\mathfrak{X}^{k+1}(S)$ and that
$\delta_{k+1}\delta_k=0$. The cohomology of this complex is denoted
by $\HP^\bullet(S)$, and is called the Poisson cohomology of $S$.

\section{Duality between  Poisson homology and  Poisson cohomology}

Let $S$ be a finite dimensional Poisson algebra. By Proposition \ref{dual}, the left Poisson module $S$ induces a right Poisson module structure on $S^*$.
For a left Poisson module $M$,  there is cochain complex $(\Hom_S(\Omega^\bullet(S), M), \delta_\bullet^\prime)$, where $\delta_k^\prime: \Hom_S(\Omega^{k}(S), M)\rightarrow \Hom_S(\Omega^{k+1}(S), M)$ is given by \begin{align*}
\delta_k^\prime(g)(da_0\wedge da_1\wedge& \cdots\wedge da_k) := \sum_{0\leq i\leq k}(-1)^i\{a_i, g(da_0\wedge \cdots\wedge \widehat{da_i}\wedge \cdots\wedge da_k)\}_M +\\
 & +\sum_{0\leq i< j\leq k}(-1)^{i+j}g(d\{a_i, a_j\}\wedge da_0\wedge \cdots\wedge \widehat{da_i}\wedge \cdots\wedge\widehat{da_j}\wedge\cdots\wedge da_k)
\end{align*} for all $g\in \Hom_S(\Omega^k(S), M)$.
We  have the following:

\begin{proposition} \label{dual3}  Let $M$ be a finite dimensional left Poisson module. The following diagram is commutative for any $k\in \mathbb{N}$:
$$\xymatrix{
 \Hom_S(\Omega^k(S), M)\ar[r]^{\varphi} \ar[d]^{\delta_k^\prime}
                &\Hom_\mathbbm{k}(M^*\otimes_S \Omega^k(S), \mathbbm{k}) \ar[d]^{\partial_k^*}  \\
\Hom_S(\Omega^{k+1}(S), M)\ar[r]^-{\varphi}
                &\Hom_\mathbbm{k}(M^*\otimes_S \Omega^{k+1}(S), \mathbbm{{k}}),} $$
where $\varphi$ is the canonical isomorphism $f\mapsto \varphi_f$ defined by  $\varphi_f(\alpha\otimes x):=\alpha(f(x))$ for
any $x\in \Omega^k(S)$ and $\alpha\in M^*$.
\end{proposition}
\begin{proof} It suffices to show that $\varphi\delta_k^\prime(f)$ and $\partial_k^*\varphi(f)$ define the same morphism from  $M^*\otimes_S \Omega^{k+1}(S)$ to $\mathbbm{k}$ for any $f\in \Hom_S(\Omega^k(S), M)$. Indeed, for any $x=da_0\wedge da_1\wedge \cdots\wedge da_k\in \Omega^{k+1}(S)$ and $\alpha\in M^*$, we have:
\begin{align*}
&\varphi\delta_k^\prime(f)(\alpha \otimes x)\\
=&\alpha(\delta_k^\prime(f)(x))\\
= &\sum_{0\leq i\leq k}(-1)^i\alpha(\{a_i, f(da_0\wedge \cdots\wedge \widehat{da_i}\wedge \cdots\wedge da_k)\}_M) +\\
 & +\sum_{0\leq i< j\leq k}(-1)^{i+j}\alpha f(d\{a_i, a_j\}\wedge da_0\wedge \cdots\wedge \widehat{da_i}\wedge \cdots\wedge\widehat{da_j}\wedge\cdots\wedge da_k).
\end{align*}
On the other hand, we have:
\begin{align*}
&\partial_k^*\varphi(f)(\alpha \otimes x)\\
=&\varphi_f(\partial_k(\alpha \otimes x))\\
= &\varphi_f\Big(\sum_{0\leq i\leq k}(-1)^{i}\{\alpha, a_i\}_{M^*}\otimes da_0\wedge\cdots \widehat{da_i}\wedge\cdots \wedge da_k +\\
 & +\sum_{0\leq i< j\leq k}(-1)^{i+j}\alpha\otimes d\{a_i, a_j\}\wedge da_0\wedge\cdots \wedge\widehat{da_i}\wedge\cdots \wedge\widehat{da_j}\wedge\cdots \wedge da_k\Big)\\
= &\{\alpha, a_i\}_{M^*}\Big(f(\sum_{0\leq i\leq k}(-1)^{i}da_0\wedge\cdots \widehat{da_i}\wedge\cdots \wedge da_k)\Big) +\\
 & +\alpha\Big(f(\sum_{0\leq i< j\leq k}(-1)^{i+j}d\{a_i, a_j\}\wedge da_0\wedge\cdots \wedge\widehat{da_i}\wedge\cdots \wedge\widehat{da_j}\wedge\cdots \wedge da_k)\Big).
\end{align*}
Hence, $\varphi\delta_k^\prime(f)(\alpha \otimes x)=\partial_k^*\varphi(f)(\alpha \otimes x)$ by the Poisson module structure on $M^*$.
\end{proof}

Now we can state the following duality:

\begin{theorem} \label{fdim}Let $S$ be a finite dimensional Poisson algebra. Then, for all $i\in \mathbb{N}$, $$\HP^{i}(S, \,S)^*\cong\HP_i(S, \,S^*).$$
\end{theorem}
\begin{proof}
Let $M=S$. It is easy to see that
the cochain complex $(\Hom_S(\Omega^\bullet(S), M),\delta_\bullet^\prime)$ coincides with the cochain complex $(\mathfrak{X}^\bullet(S),\delta_\bullet)$,
where the latter is used to compute the Poisson cohomology of $S$.
Since $\Hom_\mathbbm{k}(-, \mathbbm{k})$ is an exact functor, it commutes with the cohomological functors.  Thus, by taking the cohomology of the columns in the commutative diagram in Proposition \ref{dual3}, we obtain the isomorphism between Poisson homology and Poisson cohomology.
\end{proof}

Suppose that $S$ is a Frobenius Poisson algebra. By definition we have
an isomorphism of right $S$-modules $\sigma: S\rightarrow S^*$. By Proposition \ref{dual}, $S^*$ is a right Poisson module.  This right Poisson module structure on $S^*$ induces a right Poisson module structure on $S$, see Proposition \ref{iso}.
In order to distinguish this right Poisson module structure from the regular right Poisson module structure on $S$, we call it the induced Poisson module structure, and denote it by $S_\sigma$. Note that a Frobenius Poisson algebra is always symmetric as an algebra. Hence, $S_\sigma$ and $S$ are isomorphic as $S$-modules.
Following Theorem \ref{fdim}, we obtain the following duality:

\begin{corollary} \label{dual2}Let $S$ be a Frobenius Poisson algebra. Then,
for all $i\in \mathbb{N}$,  we have:
$$\HP^{i}(S, \, S)^*\cong\HP_i(S, \,S_\sigma).$$
\end{corollary}

Below, we consider two examples of Frobenius Poisson algebras.

\begin{example}\label{ex3.4}
First, let $\Lambda$ be the Poisson algebra defined on algebra
$$\mathbbm{k}[x_1, \cdots, x_n]/(x_i^2, 1\leq i\leq n)$$ with
the Poisson bracket given by $\{x_i, x_j\}=x_ix_j$ for $i<j$. The Poisson algebra $\Lambda$  is of dimenion $2^n$ with a basis
$\{1\}\bigcup\{x_{i_1}\cdots x_{i_k}|i_1 < \cdots < i_k, \{i_1, \cdots, i_k\}\subseteq \{1, 2, \cdots, n\}\}$. 
The Frobenius isomorphism,  $\sigma:\Lambda
\rightarrow \Lambda^*$, is given by:
$$\sigma(1): \sum\limits_{k=0}^n\sum c_{i_1,\cdots, i_k}x_{i_1}\cdots x_{i_k}\rightarrow c_{1,\cdots, n},$$
 where the second sum is taken over all elements in the above basis. Thus, by Proposition \ref{dual},
the induced right Poisson module $\Lambda_\sigma$ is given by:
 $$\{x_i, x_j\}_{\Lambda_{\sigma}}:=
\begin{cases}(n+2-2j)x_ix_j, & i < j, \\
  0, & i = j, \\
  (n-2j)x_jx_i, & i > j.
\end{cases}$$
Hence, we have the isomorphism: $\HP^{i}(\Lambda, \,\Lambda)^*\cong\HP_i(\Lambda, \,\Lambda_\sigma)$, for all $i\in \mathbb{N}$.
\end{example}

\begin{example}\label{ex3.5}
In  \cite{LR09}, Launois and Richard study the Poisson algebra
$$\Lambda(a, b):=\mathbbm{k}[x, y]/(x^a, y^b)$$
with the Poisson bracket $\{x, y\}=xy$ ($a, b\geq 2$, two integers). As an algebra, $\Lambda(a, b)$ is a symmetric Frobenius algebra.
By viewing $\Lambda(a, b)$ as a semi-classical
limit of the quantum complete intersection, the authors constructed a Poisson module $\Lambda(a, b)_\sigma$ and computed explicitly
the Poisson (co)homology group point-wise. Their duality theorem for $\Lambda(a,b)$ now follows from Corollary \ref{dual2}:

$$\HP^{i}(\Lambda(a, b), \,\Lambda(a, b))^*\cong\HP_i(\Lambda(a, b), \,\Lambda(a, b)_\sigma), \text{for all} \,\,i\in \mathbb{N}.$$
\end{example}

Note that the above Poisson (co)homology group  vanishes when $i\geq 3$.  The Frobenius isomorphism $\sigma:\Lambda(a, b)\rightarrow \Lambda(a, b)^*$ of right $\Lambda(a, b)$-modules  is defined by:
$$\sigma(1): \sum_{\scriptstyle 0\leq i \leq a-1 \atop\scriptstyle  0\leq j\leq b-1}c_{i, j}x^{i}y^{j}\rightarrow c_{a-1, b-1}.$$
The induced Poisson module $\Lambda(a, b)_\sigma$ is given by: 
$$\{x^iy^j, x\}_{\Lambda(a, b)_{\sigma}}=(b-1-j)x^{i+1}y^j, \quad\quad \{x^iy^j, y\}_{\Lambda(a, b)_{\sigma}}=-(a-1-i)x^{i}y^{j+1}.$$

\section{BV-structure}

In this section, we investigate the Batalin-Vilkovisky structure on the Poisson cohomology ring of a Frobenius Poisson algebra. The key point in this discussion is the non-degenerated bilinear form which defines the Frobenius structure.

\begin{definition}
A Gerstenhaber algebra $(A^\bullet, \wedge, [\,\,, \,\,])$ is a $\mathbb{Z}$-graded commutative algebra $(A, \wedge)$,
together with a bracket $[\,\,, \,\,]: A\otimes A\rightarrow A$ of degree $-1$ such that the induced bracket of degree
zero on the shifted graded space $A^{\bullet+1}$ is a graded Lie algebra, satisfying the compatible condition:
$$[a\wedge b, c]=a\wedge[b, c]+(-1)^{(m-1)n} b\wedge[a, c], a\in A^m, b\in A^n, c\in A.$$
\end{definition}

\begin{definition}
A Batalin-Vilkovisky (BV, for short) algebra is a triple $(A^\bullet, \wedge, \Delta)$, where $(A, \wedge)$ is a $\mathbb{Z}$-graded commutative algebra,
and  $\Delta: A\rightarrow A$ is an operator of degree $-1$ such that $\Delta^2=0$, and  the bracket $[\,\,, \,\,]$ defined by:
\begin{equation} \label{bvdef}
-(-1)^{(m-1)n}[a, b]:=\Delta(a\wedge b)-\Delta(a)\wedge b-(-1)^ma\wedge\Delta(b)+(-1)^ma\wedge\Delta(1)\wedge b
\end{equation} for $a\in A^m, b\in A^n$, endows $(A^\bullet, \wedge, [\,\,, \,\,])$ with a Gerstenhaber algebra structure. Here, $1\in A^0$ is the identity of the algebra.
\end{definition}

\begin{remark} If there exists a degree $-1$ operator $\Delta$ with $\Delta^2=0$ on a Gerstenhaber algebra $(A^\bullet, \wedge, [\,\,, \,\,])$ such that the equation \eqref{bvdef} holds, then $A$ is a BV algebra.
\end{remark}

\begin{example}
The Hochschild cohomology ring $\HH(A)$ of an associative algebra is a Gerstenhaber algebra \cite{Ger63}.
If $A$ is a Calabi-Yau or a symmetric Frobenius algebra,  $\HH(A)$  has a BV-structure \cite{Gin06, Tra08, ES09}.
\end{example}

Let $S$ be a Poisson algebra and $\mathfrak{X}(S):=\oplus_i\mathfrak{X}^i(S)$.
The wedge product $\wedge: \mathfrak{X}^m(S)\otimes \mathfrak{X}^n(S)\rightarrow \mathfrak{X}^{m+n}(S)$  is given by:
$$(P\wedge Q)(a_1, \cdots, a_{m+n}):=\sum\limits_{\theta\in S_{m, n}}\varepsilon(\theta) P(a_{\theta_1}, \cdots, a_{\theta_m})Q(a_{\theta_{m+1}}, \cdots, a_{\theta_{m+n}}),$$
where $S_{m, n}$ is the set of $(m, n)$-shuffle in the symmetric group $S_{m+n}$, i.e., the set of permutations
$\theta\in S_{m+n}$ satisfying $\theta_1<\cdots<\theta_m$  and $\theta_{m+1}<\cdots<\theta_{m+n}$. One has
the Schouten bracket: $$[P, Q]_S:=P\circ Q-(-1)^{(m-1)(n-1)}Q\circ P,$$
where the product $\circ$ is defined by:
$$(P\circ Q)(a_1, \cdots, a_{m+n-1}):=\sum\limits_{\theta\in S_{m-1,n}}\varepsilon(\theta) P(Q(a_{\theta_1}, \cdots, a_{\theta_n}), a_{\theta_{n+1}}, \cdots, a_{\theta_{m+n-1}}),$$ for all $P\in\mathfrak{X}^m(S)$ and $Q\in\mathfrak{X}^n(S)$.
The triple $(\mathfrak{X}(S), \wedge,  [\,\,, \,\,]_S)$ is a
Gerstenhaber algebra. Note that the wedge product and Schouten bracket induce two operations on the cohomology ring such that $\HP(S):=\oplus\HP^\bullet(S)$  has a Gerstenhaber algebra structure \cite{LPV13}. Such an algebra is called the Poisson cohomology ring.

\begin{theorem} \label{combv}
Let $S$ be a Frobenius Poisson algebra with a non-degenerated bilinear form $\langle-, -\rangle: S\otimes S\rightarrow \mathbbm{k}$.
For $P\in \mathfrak{X}^m(S)$\quad $(m\geq 1)$, define $\Delta(P)\in \mathfrak{X}^{m-1}(S)$ as follows:
$$\langle\Delta(P)(a_1, \cdots, a_{m-1}), a_m\rangle=(-1)^{m-1}\langle P(a_1, \cdots, a_{m-1}, a_m),1\rangle.$$
Then,  $(\mathfrak{X}(S),  \wedge, \Delta)$ is a BV algebra.
\end{theorem}
\begin{proof} It is obvious that $\Delta^2=0$. It remains to show that the following identity holds:
\begin{equation}\label{eqxx1}
[P, Q]_S=-(-1)^{(m-1)n}(\Delta(P\wedge Q)-\Delta(P)\wedge Q-(-1)^mP\wedge\Delta(Q))
\end{equation}
 for any $P\in\mathfrak{X}^m(S)$ and $Q\in\mathfrak{X}^n(S)$.
Given $\theta\in S_{m, n}$, we have either $\theta_m=m+n$ or $\theta_{m+n}=m+n$. Set:
$$(P\wedge_1Q)(a_1, \cdots, a_{m+n}):=\sum_{\scriptstyle \theta\in S_{m, n} \atop\scriptstyle  \theta_{m}=m+n}\varepsilon(\theta) P(a_{\theta_1}, \cdots, a_{\theta_m})Q(a_{\theta_{m+1}}, \cdots, a_{\theta_{m+n}}),$$
and: $$(P\wedge_2Q)(a_1, \cdots, a_{m+n}):=\sum_{\scriptstyle \theta\in S_{m, n} \atop\scriptstyle  \theta_{m+n}=m+n}
\varepsilon(\theta) P(a_{\theta_1}, \cdots, a_{\theta_m})Q(a_{\theta_{m+1}}, \cdots, a_{\theta_{m+n}}).$$
Then, $P\wedge Q=P\wedge_1Q+P\wedge_2Q$.
We claim that the following identity holds:
\begin{equation}\label{eqxx2}
P\circ Q+(-1)^{(m-1)n}(\Delta(P\wedge_1Q)-\Delta(P)\wedge Q)=0.
\end{equation}
Indeed, for any $a_1, \cdots, a_{m+n-1}, a_{m+n}\in S$, we have:
\begin{align*}
\langle P\circ Q&(a_1, \cdots, a_{m+n-1}), a_{m+n} \rangle\\
=&\sum\limits_{\eta\in S_{n, m-1}}\varepsilon(\eta)\langle P(Q(a_{\eta_1}, \cdots, a_{\eta_{n}}), a_{\eta_{n+1}}, \cdots, a_{\eta_{m+n-1}}), a_{m+n}\rangle\\
=&\sum\limits_{\eta\in S_{n, m-1}}\varepsilon(\eta) \langle P(Q(a_{\eta_1}, \cdots, a_{\eta_{n}}), a_{\eta_{n+1}}, \cdots, a_{\eta_{m+n-1}})a_{m+n}, 1\rangle;
\end{align*}
and:
\begin{align*}
\langle\Delta(P\wedge_1 Q)&(a_1, \cdots, a_{m+n-1}), a_{m+n}\rangle\\
=&(-1)^{m+n-1}\langle(P\wedge_1 Q)(a_1, \cdots, a_{m+n-1}, a_{m+n}), 1\rangle\\
=&(-1)^{m+n-1} \sum_{\scriptstyle \tau\in S_{m, n} \atop\scriptstyle  \tau_{m}=m+n}\varepsilon(\tau)\langle P(a_{\tau_1}, \cdots, a_{\tau_{m}})Q(a_{\tau_{m+1}}, \cdots, a_{\tau_{m+n}}), 1\rangle.
\end{align*}
Moreover,
\begin{align*}
\langle(\Delta(P)\wedge Q)&(a_1, \cdots, a_{m+n-1}), a_{m+n} \rangle\\
=&\sum\limits_{\theta\in S_{m-1, n}}\varepsilon(\theta)\langle \Delta(P)(a_{\theta_1}, \cdots, a_{\theta_{m-1}})Q(a_{\theta_{m}}, \cdots, a_{\theta_{m+n-1}}), a_{m+n}\rangle\\
=&\sum\limits_{\theta\in S_{m-1, n}}\varepsilon(\theta)\langle \Delta(P)(a_{\theta_1}, \cdots, a_{\theta_{m-1}}), Q(a_{\theta_{m}}, \cdots, a_{\theta_{m+n-1}})a_{m+n}\rangle\\
=&\sum\limits_{\theta\in S_{m-1, n}}(-1)^{m-1}\varepsilon(\theta)\langle P(a_{\theta_1}, \cdots, a_{\theta_{m-1}}, Q(a_{\theta_{m}}, \cdots, a_{\theta_{m+n-1}})a_{m+n}), 1\rangle\\
=&\sum\limits_{\theta\in S_{m-1, n}}(-1)^{m-1}\varepsilon(\theta)\langle P(a_{\theta_1}, \cdots, a_{\theta_{m-1}}, Q(a_{\theta_{m}}, \cdots, a_{\theta_{m+n-1}}))a_{m+n}, 1\rangle\\
&+\!\!\sum\limits_{\theta\in S_{m-1, n}}\!\!\!(-1)^{m-1}\varepsilon(\theta)\langle P(a_{\theta_1}, \cdots, a_{\theta_{m-1}}, a_{m+n})Q(a_{\theta_{m}}, \cdots, a_{\theta_{m+n-1}}), 1\rangle.
\end{align*}

Now for $r$ distinct integers
$i_1, \cdots, i_r$,  let $(i_1, \cdots, i_r)$ be the cyclic
permutation sending $i_1$ to $i_2$, $\cdots$, $i_{r-1}$ to $i_r$ and
$i_r$ to $i_1$. Then, the map $\theta\in
S_{m-1, n}\mapsto \tau\in \{\tau\in S_{m, n}|\tau_m=m+n\}$ given by $\tau:=\theta\circ(m+n, m+n-1, \cdots\, m)$ is
well-defined and is a bijection. In fact, it is easy to see that these
two sets have the same cardinality  and
the map is injective.  Moreover, one has
$\varepsilon(\tau)=(-1)^{n}\varepsilon(\theta)$. Hence, by the change of variables, we obtain:
\begin{align*}&(-1)^{m-1}\sum\limits_{\theta\in S_{m-1, n}}\varepsilon(\theta)P(a_{\theta_1}, \cdots, a_{\theta_{m-1}}, a_{m+n})Q(a_{\theta_{m}}, \cdots, a_{\theta_{m+n-1}})\\=&(-1)^{m+n-1} \sum_{\scriptstyle \tau\in S_{m, n} \atop\scriptstyle  \tau_{m}=m+n}\varepsilon(\tau)P(a_{\tau_1}, \cdots, a_{\tau_{m}})Q(a_{\tau_{m+1}}, \cdots, a_{\tau_{m+n}}).\end{align*}
Consider the map $\theta\in
S_{m-1, n}\mapsto \eta\in S_{n, m-1}$ defined by: $$\eta:=\theta\circ\Big(\begin{array}{cccccccc}
  1 & 2 & \cdots & n & n+1 & n+2 &  \cdots & m+n-1 \\
  m & m+1 & \cdots & m+n-1 & 1 & 2 & \cdots & m-1
\end{array}\Big).$$  It is easy to see that the foregoing map is
well-defined and is bijective. Furthermore, we have $\varepsilon(\eta)=(-1)^{n(m-1)}\varepsilon(\theta)$.
Hence,
\begin{align*}&\sum\limits_{\theta\in S_{m-1, n}}(-1)^{m-1}\varepsilon(\theta) P(a_{\theta_1}, \cdots, a_{\theta_{m-1}}, Q(a_{\theta_{m}}, \cdots, a_{\theta_{m+n-1}}))a_{m+n}\\
=&\sum\limits_{\eta\in S_{n, m-1}}(-1)^{(m-1)(n+1)}\varepsilon(\eta) P(a_{\eta_{n+1}}, \cdots, a_{\eta_{m+n-1}}, Q(a_{\eta_1}, \cdots, a_{\eta_{n}}))a_{m+n}.
\end{align*}
Now the equation \eqref{eqxx2} follows from the fact that $P$ is a skew-symmetric derivation and  that the bilinear form is non-degenerated.

Using a similar argument to the above one and the symmetry of the bilinear form $\langle-, -\rangle$, i.e., $\langle ab, c\rangle=\langle b, ca\rangle$, one can obtain the second identity:
\begin{equation}\label{eqxx3}
Q\circ P-(-1)^{m-1}\Delta(P\wedge_2Q)-P\wedge \Delta(Q)=0.
\end{equation}
Now combining Equation \eqref{eqxx2} and Equation \eqref{eqxx3}, one obtains Equation \eqref{eqxx1}.
\end{proof}

By definition, the operator $\Delta$ satisfies the following commutative diagram:
$$\xymatrix{
\mathfrak{X}^k(S)\ar[d]^{\Delta} \ar[r]^{\bigstar}
                & (\Omega^{k}(S))^* \ar[d]^{\mathbf{d}^*}  \\
\mathfrak{X}^{k-1}(S) \ar[r]^-{\bigstar}
                &(\Omega^{k-1}(S))^*      }$$
where $\mathbf{d}^*$  is the dual of de Rham differential and $\bigstar$ is the canonical isomorphism defined by $$\bigstar(P)(a_0da_1\cdots da_k):=\langle a_0, P(a_1, \cdots, a_k)\rangle.$$
Note that $\Delta$ is similar to the divergence operator with respect to
a volume form on an orientable Poisson manifold in the smooth case.

%
For an orientable Poisson manifold $M$ with the Poisson bivector field $\pi$, the divergence of $\pi$ is minus
the modular vector field of $\pi$ which is a Poisson $1$-cocycle and the cohomology class of such a $1$-cocycle is independent
of the chosen volume form.  If this cohomology class is trivial, then $M$ is said to be a  unimodular Poisson manifold.
Stimulated by the above idea, we introduce the notion of a unimodular Frobenius Poisson algebra.

\begin{definition} Let $S$ be a Frobenius Poisson algebra with the Poisson structure $\pi$, and let $\Delta$ be the operator defined as above. If $\Delta(\pi)=0$, then $S$ is called a unimodular Poisson algebra.
\end{definition}

\begin{proposition} \label{uni} For a Frobenius Poisson algebra $S$, the following are equivalent:
\begin{itemize}
\item[(1)] $S$ is unimodular;
\item[(2)] $\langle \{a, b\}, c\rangle=\langle b, \{c, a\}\rangle$ for any $a, b, c\in S$;
\item[(3)] the Frobenius isomorphism $\sigma: S\rightarrow S^*$ is a right Poisson module isomorphism.
\end{itemize}
\end{proposition}
\begin{proof} $(1)\Longrightarrow(2)$.  By definition, the unimodularity  of $S$ is equivalent to  the equation $\langle \{x, y\}, 1\rangle=0$  for any $x, y \in S$.  Now  if $S$ is unimodular, then we have:
 \begin{align*} \langle \{a, b\}, c\rangle-\langle b, \{c, a\}\rangle&=\langle \{a, b\}c, 1\rangle-\langle \{c, a\}b, 1\rangle\\
&=\langle \{a, b\}c, 1\rangle+\langle \{a, c\}b, 1\rangle \\
&=\langle \{a, bc\}, 1\rangle=0.\end{align*}
$(2)\Longrightarrow(1)$ is  obvious.
For the equivalence $(2)\Longleftrightarrow(3)$, we note that $\sigma: S\rightarrow S^*$ is a right Poisson module homomorphism if and only if
$\sigma(\{x, y\}_S)=\{\sigma(x), y\}_{S^*}$
holds in $S^*$ for any $x, y\in S$. That is, $$\langle \{y, z\}, x\rangle=\langle z, \{x, y\}\rangle$$
for any $x, y, z\in S$. Hence, $(2)\Longleftrightarrow(3)$ follows.
\end{proof}

\begin{remark}
\begin{itemize}
\item[(1)] The algebras $\Lambda$ and $\Lambda(a,b)$ described in Example \ref{ex3.4} and Example \ref{ex3.5} are not unimodular.
\item[(2)] For a unimodular Poisson algebra $S$, the induced Poisson module $S_\sigma$ is isomorphic to $S$ as right Poisson modules. It follows from  Theorem \ref{dual2} that the duality below holds:
$$\HP^{i}(S, \, S)^*\cong\HP_i(S, \,S), \quad \forall i\in \mathbb{N}.$$
\end{itemize}\end{remark}

\begin{example} Let $S$ be the Poisson algebra on $$\mathbbm{k}[x, y,z]/(x^2, y^2, z^2),$$
with the Poisson bracket defined by:
\begin{align*} \{x, y\}:&=xy,\\
\{y, z\}:&=yz,\\
\{x, z\}:&=-xz.
\end{align*}
Then, $S$ is a unimodular Poisson algebra.
\end{example}

\begin{theorem} \label{pcombv}
Let $S$ be a unimodular Poisson algebra. Then $\Delta$ induces an operator on Poisson cohomology $\HP^\bullet(S)$ whose square is zero. Furthermore, for
$\alpha\in \HP^m(S)$, $\beta\in \HP^n(S)$, the following identity holds,
$$[\alpha, \beta]_S=-(-1)^{(m-1)n}(\Delta(\alpha\wedge \beta)-\Delta(\alpha)\wedge\beta-(-1)^m\alpha\wedge\Delta(\beta)).$$
Equivalently, the Poisson cohomology ring $\HP(S)$ has a BV algebra structure.
\end{theorem}
\begin{proof} First of all, we show that the operator $\Delta$ commutes with the coboundary operator $\delta$ on the cochain complex $(\mathfrak{X}^\bullet(S),\delta_\bullet)$. It is clear that $\Delta\delta(P)=0$ for $P\in \mathfrak{X}^0(S)$. Now, for each $P\in \mathfrak{X}^m(S)(m\geq 1)$, we claim that
$(\delta\Delta+\Delta\delta)(P)=0$. Indeed, for any $a_0, \cdots, a_{m-1}, x\in S$, we have:
\begin{align*} \langle \Delta\delta(P)&(a_0, \cdots, a_{m-1}), x\rangle\\
=&(-1)^{m}\langle\delta(P)(a_0, \cdots, a_{m-1}, x),1\rangle\\
=&(-1)^{m}\sum_{0\leq i\leq m-1}(-1)^i\langle\{a_i, P(a_0, \cdots,\widehat{ a_i},  \cdots, a_{m-1}, x)\}, 1\rangle\\
&+(-1)^{m}(-1)^m\langle\{x, P(a_0, \cdots, a_{m-1}, x)\}, 1\rangle\\
&+(-1)^{m}\sum_{0\leq i< j\leq m-1}(-1)^{i+j}\langle P(\{a_i, a_j\}, a_0, \cdots, \widehat{a_i}, \cdots\widehat{a_j}, \cdots, a_{m-1}, x), 1\rangle\\
&+(-1)^{m}\sum_{0\leq i\leq m-1}(-1)^{i+m}\langle P(\{a_i, x\}, a_0, \cdots, \widehat{a_i}, a_{m-1}), 1\rangle.
\end{align*}
The first two terms of the above sum are equal to zero because of the identity in Proposition \ref{uni} (2).
On the other hand, we have:
\begin{align*} \langle \delta\Delta(P)&(a_0, \cdots, a_{m-1}), x\rangle\\
=&\sum_{0\leq i\leq m-1}(-1)^i\langle \{a_i, \Delta(P)(a_0, \cdots,\widehat{ a_i},  \cdots, a_{m-1})\}, x\rangle\\
&+\sum_{0\leq i< j\leq m-1}(-1)^{i+j}\langle \Delta(P)(\{a_i, a_j\}, a_0, \cdots, \widehat{a_i}, \cdots\widehat{a_j}, \cdots, a_{m-1}), x\rangle\\
=&\sum_{0\leq i\leq m-1}(-1)^i\langle \Delta(P)(a_0, \cdots,\widehat{ a_i},  \cdots, a_{m-1}), \{x, a_i\}\rangle \quad\text{by \,\,Proposition} \,\ref{uni} (2)\\
&+\sum_{0\leq i< j\leq m-1}(-1)^{i+j}(-1)^{m-1}\langle P(\{a_i, a_j\}, a_0, \cdots, \widehat{a_i}, \cdots\widehat{a_j}, \cdots, a_{m-1}, x), 1\rangle.
\end{align*}
It follows that:
$$\langle \Delta\delta(P)(a_0, \cdots, a_{m-1}), x\rangle=-\langle \delta\Delta(P)(a_0, \cdots, a_{m-1}), x\rangle.$$
Thus, by the non-degeneracy of the bilinear form, one obtains $\Delta\delta(P)=-\delta\Delta(P)$.
Therefore, $\Delta$ induces an operator on the Poisson cohomology $\HP^\bullet(S)$. The rest is clear
in view of Theorem \ref{combv}.
\end{proof}

\section*{Acknowledgment}
This work is supported by Natural Science Foundation of China \#11201299 and by an FWO grant.

\bibliography{}

\begin{thebibliography}{MM}
\bibitem [Bry88]{Bry88} J. L.  Brylinski,  A differential complex for Poisson
manifolds,  J. Differential Geom. 28 (1988),  93--114.

\bibitem [CE56]{CE56} H. Cartan, S. Eilenberg, Homological algebra, Princeton
University Press, Princeton, 1956.

\bibitem [Dol09]{Dol09} V. A. Dolgushev,  The Van den Bergh duality and the modular symmetry
of a Poisson variety, Selecta Math., 14 (2009), 199--228.

\bibitem[Eu10]{Eu10} C.  Eu, The calculus structure of the Hochschild homology/cohomology of preprojective algebras of Dynkin quivers, J. Pure Applied Algebra, 214 (2010),  28--46.

\bibitem[ES09]{ES09} C. Eu, T. Schedler, Calabi-Yau Frobenius algebras,
J. Algebra 321 (2009),  774--815.

\bibitem [Ger63]{Ger63} M. Gerstenhaber, The cohomology structure of an associative ring,
Ann. of  Math., 78 (1963), 267--288.


\bibitem [Gin06]{Gin06} V. Ginzburg, Calabi-Yau algebras,  arXiv: math. AG/0612139.




\bibitem [HKR62]{HKR62} G. Hochschild,  B. Kostant, A. Rosenberg,  Differential forms
on regular affine algebras,  Trans. Amer. Math. Soc. 102 (1962),
383--408.



\bibitem [Hue90]{Hue90} J. Huebschmann, Poisson cohomology and quantization, J. Reine Angew.
Math. 408 (1990), 57--113.



\bibitem [Kas88]{Kas88} C. Kassel, L'homologie cyclique des alg\`{e}bres
enveloppantes, (French) [The cyclic homology of enveloping
algebras], Invent. Math., 91 (1988), 221--251.

\bibitem [Kon03]{Kon03} M. Kontsevich, Deformation quantization of Poisson manifolds, Lett.
Math. Phys. 66 (2003), 157--216.


%



\bibitem [LR07]{LR07} S. Launois, L. Richard, Twisted Poincar\'{e} duality for some
quadratic Poisson algebras,  Lett. Math. Phys., 79 (2007), 161--174.

\bibitem [LR09]{LR09} S. Launois, L. Richard,  Poisson
(co)homology of truncated polynomial algebras in two variables,  C.
R. Math. Acad. Sci. Paris 347 (2009),  133--138.

\bibitem [LPV13]{LPV13} c. Laurent-Gengoux,  A. Pichereau and P. Vanhaecke, Poisson structures,
 Springer, Heidelberg, 2013.

\bibitem [Lic77]{Lic77} A.  Lichnerowicz,  Les vari\'{e}t\'{e}s de Poisson
et leurs alg\`{e}bres de Lie associ\'{e}es, (French),  J.
Differential Geometry,  12 (1977),  253--300.

%
%

\bibitem [Men04]{Men04} L. Menichi, Batalin-Vilkovisky algebras and cyclic cohomology of Hopf algebras, K-Theory 32 (2004), 231--251.



\bibitem [Oh99]{Oh99} S. Q. Oh, Poisson enveloping algebras, Comm. Algebra, 27 (1999), 2181--2186.




\bibitem[Sm96]{Sm96} S. P. Smith, {Some finite dimensional algebras related to elliptic curves},
       CMS Conf. Proc. 19 (1996),  315--348.

\bibitem[Tra08]{Tra08} T. Tradler,  The Batalin-Vilkovisky algebra on Hochschild cohomology induced by infinity inner products, Ann. Inst. Fourier, 58 (2008),  2351--2379.








\bibitem[Xu99]{Xu99} P. Xu, Gerstenhaber algebras and BV-algebras in Poisson geometry,
Commun. Math. Phys., 200 (1999), 545--560.


\end{thebibliography}

\end{document}